\documentclass{amsart}
\usepackage[margin=1in]{geometry}
\usepackage{amsmath,amssymb,amsthm}
\usepackage{hyperref}
\usepackage{xcolor}
\usepackage{graphicx}

\newtheorem{theorem}{Theorem}[section]
\newtheorem{proposition}[theorem]{Proposition}
\newtheorem{lemma}[theorem]{Lemma}

\theoremstyle{remark}
\newtheorem{remark}[theorem]{Remark}

\newcommand{\T}{\mathbb T}

\title{Uniqueness of synchronized stationary equilibria in the Kuramoto mean field game}
\author{Sebastian Munoz}
\address{Department of Mathematics, University of California, Los Angeles, Los Angeles, CA 90095, USA}
\email{sebastian@math.ucla.edu}
\date{}

\begin{document}

\begin{abstract}
The Kuramoto mean field game describes a population of 
oscillators which, above a critical value of the
interaction parameter, form synchronized Nash equilibria. We show that these synchronized equilibria are unique, up to rotations, for each value of the parameter. Moreover, the unique equilibrium depends
smoothly on the interaction parameter, and converges smoothly to the uniform
distribution when the parameter approaches its critical value. These statements are direct consequences of our main technical result: the self-consistency map is strictly concave, as was previously 
conjectured by Carmona, Cormier, and Soner \cite{CCS}. In the proof, we decompose the second derivative
of the self-consistency map as the sum of two moments, and prove that each of them is negative by using symmetrization arguments to exploit hidden monotonicity and skewness properties of the equilibria.
\end{abstract}

\maketitle

\section{Introduction}\label{sec:intro}

Mean field games (MFG) are non-cooperative differential games with
infinitely many players~\cite{LL,HMC,CarmonaDelarue}, in which
the players determine an optimal strategy by observing the distribution of the others. 

In non-monotone MFG, several nontrivial Nash equilibria
may exist, and it is interesting to understanding how the set of equilibria depends on the parameters of the model. We resolve this question for
the Kuramoto MFG proposed by Carmona, Cormier, and
Soner \cite{CCS}, a MFG
model of phase synchronization.

The classical Kuramoto
model \cite{Kuramoto,Strogatz,ABPVS} describes a population of coupled
oscillators that is incoherent below a critical interaction strength
and, and self-organizes into a synchronized configuration when the threshold is surpassed. In the classical setting, where the population evolves according to
prescribed McKean--Vlasov dynamics, the phase transition and the
stability of the incoherent state have been studied in detail through both the nonlinear Fokker--Planck equation \cite{GPP} and through the
fluctuations of the underlying particle system \cite{BGP}.

An MFG version of the model,
in which each oscillator chooses its dynamics to minimize a cost, was studied
formally in \cite{YMMS} and used in \cite{CG} for an analysis of
jet-lag recovery (see also\cite{Cirant}). The first
rigorous study of the MFG model was
done in \cite{CCS}, where the authors proved that the
synchronization phase transition occurs in the MFG
framework as a bifurcation of the Nash equilibria, with the full
synchronization occurring when the interaction parameter becomes very large (see also \cite{CCS2,HS}).

In the model we study here, the dynamics of a typical oscillator correspond to a controlled
diffusion on the circle $\T:=\mathbb R/2\pi\mathbb Z$, namely
\[
dX_t=\alpha_t\,dt+\sigma\,dB_t,
\]
where $B_t$ is a Brownian motion. For a fixed population distribution
$\mu_t\in\mathcal P(\T)$, a typical agent chooses the drift $\alpha$ with the goal of minimizing the
cost functional
\[
J(\alpha):=\mathbb E\left[\int_0^\infty e^{-\beta t}
\Bigl(\kappa\,c(X_t,\mu_t)+\tfrac12\alpha_t^2\Bigr)dt\right].
\]
 In the above, $\beta>0$ is the discount rate,
$\sigma>0$ is the noise level, $\kappa>0$ is the strength of the interaction, and
\[
c(\theta,\mu):=\int_\T 2\sin^2\tfrac{\theta-\theta'}{2}\,\mu(d\theta')
=1-\mu\bigl(\cos(\theta-\cdot)\bigr)
\]
is the classical Kuramoto coupling, which penalizes any misalignment of the
agent's state with the population. Given $(\mu_t)_{t\geq 0}$, we define the value
function by
\[
\Phi(t,x):=\inf_\alpha\mathbb E\left[\int_t^\infty e^{-\beta(s-t)}
\Bigl(\kappa\,c(X_s^{t,x,\alpha},\mu_s)
+\tfrac12\alpha_s^2\Bigr)ds\right],
\]
where $dX_s^{t,x,\alpha}=\alpha_s\,ds+\sigma\,dB_s$ and
$X_t^{t,x,\alpha}=x$. By dynamic programming, one then obtains the optimal feedback
$\alpha_s^*=-\Phi_x(s,X_s)$, as well as the Hamilton--Jacobi--Bellman (HJB) equation
\[
-\Phi_t+\beta\Phi-\frac{\sigma^2}{2}\Phi_{xx}
+\tfrac12\Phi_x^2=\kappa\,c(x,\mu_t).
\]

A \emph{Nash equilibrium} is then a population flow $(\mu_t)_{t\geq 0}$ such
that, when the representative agent uses the optimal feedback,
the law of the optimally controlled state coincides with the prescribed
population flow, namely
\begin{equation}\label{eq:dynamic-equilibrium}
\mathcal L(X_t^*)=\mu_t,\qquad t\geq 0.
\end{equation}

In this paper, we focus on stationary equilibria $\mu_t\equiv\mu \in \mathcal{P}(\T)$. These equilibria are the game-theoretic analogues of the
incoherent and synchronized equilibria of the classical Kuramoto model. They describe the possible long time organized states of the game.

In the stationary case, the infinite-horizon control problem is
time-homogeneous, so $\Phi(t,x)=\phi(x)$, where
$\phi$ solves the stationary
HJB equation
\[
\beta\phi-\frac{\sigma^2}{2}\phi_{xx}+\tfrac12\phi_x^2=\kappa\,c(\cdot,\mu),
\]
and the optimal feedback is then given by $\alpha_t^*=-\phi_x(X_t)$.
With $\mu_t\equiv\mu$, condition~\eqref{eq:dynamic-equilibrium}
says that $\mu$ is an invariant measure for the diffusion process
$dX_t=-\phi_x(X_t)\,dt+\sigma\,dB_t$, namely,
$X_0\sim\mu$ implies $X_t\sim\mu$ for all $t\geq 0$. The density of this invariant
measure is
\[
\mu(dx)\propto \exp(-2\phi(x)/\sigma^2)\,dx \qquad\text{on }\T.
\]
Since the coupling $c$ is rotation-invariant, any rotation of an equilibrium is also an
equilibrium. Hence, rotating to ensure $\mu(\sin)=0$ and $\mu(\cos)\geq 0$
(with $\mu(\cos)>0$ in the synchronized case), and absorbing the constant
$\kappa$ on the right-hand side of the HJB into $\phi$, the
self-consistency condition reduces to
\begin{equation}\label{eq:self-consistency}
\mu=\mu^{\kappa\,\mu(\cos)},
\end{equation}
where, for each $\gamma\geq 0$, $\mu^\gamma$ is the probability measure
on $\T$ defined analogously from $\phi^\gamma\in C^\infty(\T)$, the
unique smooth solution of
\begin{equation}\label{eq:HJB-intro}
\beta\phi^\gamma-\frac{\sigma^2}{2}\phi^\gamma_{xx}
+\tfrac12(\phi^\gamma_x)^2=-\gamma\cos x,\qquad x\in\T.
\end{equation}
With this normalization,~\eqref{eq:self-consistency} is the fixed-point equation
$\gamma=F_\kappa(\gamma)$ on $[0,\infty)$ for the scalar
\emph{self-consistency map}
\begin{equation}\label{eq:self-consistency-map}
F_\kappa(\gamma):=\kappa\,\mu^\gamma(\cos).
\end{equation}
The choice $\gamma=0$ corresponds to the uniform measure, which is therefore
a stationary equilibrium for every $\kappa$. We refer to non-uniform
stationary equilibria as \emph{synchronized}. A short computation,
carried out in Section~\ref{sec:setup}, gives
$F_\kappa(0)=0$, $F_\kappa'(0)=\kappa/\kappa_c$, and
$|F_\kappa|\leq\kappa$, where the \emph{critical interaction parameter}
is
\begin{equation}\label{eq:kappa-c}
\kappa_c:=\beta\sigma^2+\frac{\sigma^4}{2}.
\end{equation}
It was shown in~\cite{CCS} that a synchronized stationary equilibrium
exists for every $\kappa>\kappa_c$, and that any sequence of such equilibria converges
up to rotation to $\delta_0$ (full synchronization) as
$\kappa\to\infty$. In the subcritical regime, global attraction of the uniform state is established
in the smaller weak-interaction regime $\kappa<\beta\sigma^2/4$, while for all
$\kappa<\kappa_c$, they construct equilibrium flows, starting from initial
laws sufficiently close to the uniform distribution, that converge exponentially
to the uniform state. Very recently, H\"ofer~\cite{Hofer26} proved that
every equilibrium flow satisfying~\eqref{eq:dynamic-equilibrium}
converges weakly to the uniform state
for $\kappa<\kappa_c$. In particular, this result rules out synchronized stationary equilibria for subcritical $\kappa$.

The classical noisy Kuramoto model can be recovered from the MFG in the myopic limit
$\beta\to\infty$, leaving $\kappa/\beta$ fixed (see
Remark~\ref{rem:classical-limit} for further details).

 On the basis of numerical experiments, it was
conjectured in~\cite[Rem.~7.4]{CCS} that $F_\kappa$ is concave on
$[0,\infty)$ for every $\beta,\sigma,\kappa>0$. This concavity would readily imply the following: the
synchronized equilibrium would be unique up to rotation, and it would
converge to the uniform measure at the bifurcation $\kappa=\kappa_c$.

The first partial result toward this conjecture was obtained by
Cesaroni and Cirant~\cite{CesaroniCirant}, who worked with a closely
related model,
the ergodic Kuramoto MFG, in which the discount
factor $\beta>0$ of~\eqref{eq:HJB-intro} is instead replaced by an ergodic
constant. Their work shows that
the synchronized equilibrium is unique for large enough $\kappa$. The arguments used are variational and use
energy estimates which become sharp in the limit as $\kappa \to \infty$.
However, the concavity conjecture and
the behavior of the synchronization branch at the critical threshold have so far
remained open.

\begin{theorem}\label{thm:main}
Let $\beta,\sigma,\kappa>0$, let $F_{\kappa}$ be given by~\eqref{eq:self-consistency-map} and let $\kappa_c$ be given by \eqref{eq:kappa-c}. Then the following holds:
\begin{enumerate}
\item[\textup{(i)}] the function $F_\kappa$ is strictly concave on
$[0,\infty)$.
\item[\textup{(ii)}] for $\kappa\leq\kappa_c$ the stationary Kuramoto
MFG has no synchronized stationary equilibrium. For
$\kappa>\kappa_c$ there exists, up to rotations, a unique
synchronized stationary equilibrium $\mu_\kappa$. Moreover, the density
of $\mu_\kappa$ depends smoothly on $\kappa\in(\kappa_c,\infty)$
as a map into $C^\infty(\T)$, and converges, as $\kappa\downarrow\kappa_c$, to the uniform density
$1/(2\pi)$ in $C^\infty(\T)$.
\end{enumerate}
\end{theorem}
Part~(i) is the main analytical result of the theorem: it resolves the
concavity conjecture of~\cite[Rem.~7.4]{CCS}. Part~(ii) is a simple consequence: the strictly concave function
$g_\kappa(\gamma):=F_\kappa(\gamma)-\gamma$ satisfies $g_\kappa(0)=0$,
$g_\kappa'(0)=\kappa/\kappa_c-1$, and $g_\kappa(\gamma)\to-\infty$ as
$\gamma\to\infty$, so it has exactly one positive zero
$\gamma_*(\kappa)$ for $\kappa>\kappa_c$, and this zero tends to $0$
as $\kappa\downarrow\kappa_c$. This is explained at the end of Section~\ref{sec:proof}.

We expect that the same method, with minor modifications, will
solve the analogous uniqueness question for the ergodic Kuramoto
MFG of~\cite{CesaroniCirant}. The argument
should in fact be shorter in that setting, since the cubic
moment discussed below is absent in the ergodic problem.

The rest of the paper is devoted to the proof of part~(i).

\subsection{\texorpdfstring{Overview of the proof}{Overview of the proof}}\label{subsec:proof-overview}

We first remove two parameters from the HJB equation. Setting
\[
\lambda:=\frac{2\beta}{\sigma^2},\qquad
\zeta:=\frac{4\gamma}{\sigma^4},\qquad
v^\zeta:=\frac{2}{\sigma^2}\phi^\gamma,
\]
we obtain from~\eqref{eq:HJB-intro}
\begin{equation}\label{eq:HJB}
-v^\zeta_{xx}+\lambda v^\zeta+\tfrac12(v^\zeta_x)^2=-\zeta\cos x,
\qquad x\in\T.
\end{equation}
We denote the law $\mu^\gamma$ by $\mu^\zeta$ in these variables, so that
\begin{equation}\label{eq:mu}
\mu^\zeta(dx)=f^\zeta(x)\,dx,
\qquad
f^\zeta(x):=\frac{e^{-v^\zeta(x)}}{\mathcal Z^\zeta},
\qquad
\mathcal Z^\zeta:=\int_\T e^{-v^\zeta(y)}\,dy,
\end{equation}
and we let
\begin{equation}\label{eq:A}
A(\zeta):=\mu^\zeta(\cos)=\int_\T \cos x\,f^\zeta(x)\,dx.
\end{equation}
Since
\begin{equation}\label{eq:F-rescaling}
F_\kappa(\gamma)=\kappa\,A\!\left(\frac{4\gamma}{\sigma^4}\right),
\end{equation}
it is enough to prove that $A''(\zeta)<0$ for every $\zeta>0$.
Existence, uniqueness, and smooth dependence of $v^\zeta$ on $\zeta$
are shown in Section~\ref{sec:setup}.

We fix $\lambda>0$, and omit the superscript $\zeta$ when there is
no ambiguity. Let
\begin{equation}\label{eq:w-z-intro}
u:=v_x,\qquad w:=A/\lambda+v_\zeta,\qquad z:=v_{x\zeta}=w_x.
\end{equation}
Differentiating~\eqref{eq:HJB} in $\zeta$ and integrating by parts, we obtain (see Section~\ref{subsec:second-derivative})
\begin{equation}\label{eq:second-derivative-intro}
A''(\zeta)=-\lambda\int_\T w^3f-3\int_\T wz^2f,
\qquad \int_\T wf=0.
\end{equation}
Since $\lambda>0$, it therefore suffices to show that, for all $\zeta>0$,
\begin{equation}\label{eq:moment-signs}
\int_\T wz^2f>0,\qquad \int_\T w^3f\geq0.
\end{equation}
These are the gradient and cubic moments, respectively. We first
consider the gradient moment. Since $w$ has zero mean under $f$
by~\eqref{eq:second-derivative-intro}, we study how the weight $z^2$
changes the balance between the positive and negative contributions
of $w$. The properties that we discover will also be useful for the cubic moment.

By Proposition~\ref{prop:smooth-zeta}, $w$, $z^2$, and $f$ are
even, so it suffices to analyze the integrals on
$[0,\pi]$.

We start from the explicit solution
$v^0\equiv0$, with $f^0=1/(2\pi)$.
We substitute the Taylor expansion $v^\zeta=\zeta v_1+\zeta^2v_2+O(\zeta^3)$
into~\eqref{eq:HJB} and compare the coefficients of $\zeta$ and
$\zeta^2$, obtaining $-v_1''+\lambda v_1=-\cos x$ and
$-v_2''+\lambda v_2=-\frac12(v_1')^2$. Solving these linear ODEs gives
\begin{equation}\label{eq:v-small-zeta}
v^\zeta(x)=-\frac{\zeta\cos x}{1+\lambda}
+\frac{\zeta^2}{4(1+\lambda)^2}
\left(\frac{\cos 2x}{4+\lambda}-\frac1\lambda\right)
+O(\zeta^3).
\end{equation}
We differentiate~\eqref{eq:v-small-zeta} in $x$ to obtain $u$,
and then in $\zeta$ to obtain $z$:
\begin{equation}\label{eq:u-z-small-zeta}
\begin{aligned}
u(x)&=\frac{\zeta\sin x}{1+\lambda}
-\frac{\zeta^2\sin x\cos x}{(1+\lambda)^2(4+\lambda)}+O(\zeta^3),\\
z(x)&=\frac{\sin x}{1+\lambda}
-\frac{2\zeta\sin x\cos x}{(1+\lambda)^2(4+\lambda)}+O(\zeta^2).
\end{aligned}
\end{equation}

This suggests dividing by the common factor $\sin x$, so that, letting
\begin{equation}\label{eq:Z-definition}
U:=\frac{u}{\sin x},\qquad
Z:=\frac{z}{\sin x}=\partial_\zeta U,\qquad 0<x<\pi,
\end{equation}
we get
\begin{equation}\label{eq:quotients-small-zeta}
\begin{aligned}
U&=\frac{\zeta}{1+\lambda}
-\frac{\zeta^2\cos x}{(1+\lambda)^2(4+\lambda)}+O(\zeta^3),\\
Z&=\frac{1}{1+\lambda}
-\frac{2\zeta\cos x}{(1+\lambda)^2(4+\lambda)}+O(\zeta^2).
\end{aligned}
\end{equation}
For small positive $\zeta$, these expansions suggest that $U$ and $Z$
are positive and increasing on $(0,\pi)$. In fact,
numerical simulations indicate clearly that these properties persist
well beyond the small-$\zeta$ regime, and we prove them rigorously
in Section~\ref{subsec:shape} for every $\zeta>0$.

The monotonicity of $U$ and $Z$ means that $u$ and $z$ lean toward
$\pi$. This rightward asymmetry is illustrated in
Figure~\ref{fig:shape-comparison}, where we compare the two profiles
with the symmetric sine curve.

\begin{figure}[htbp]
\centering
\includegraphics[width=0.94\textwidth]{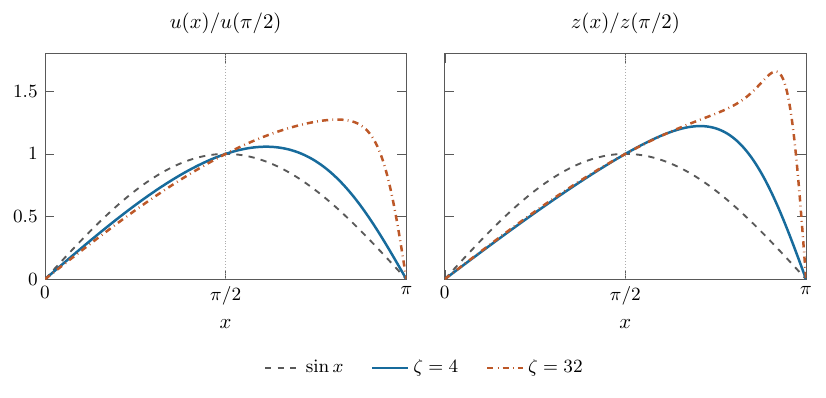}
\caption{The functions $u/u(\pi/2)$ (left) and $z/z(\pi/2)$ (right)
for $\lambda=1$ and $\zeta=4,32$, compared with $\sin x$.
Each function is normalized to equal $1$ at $x=\pi/2$.}
\label{fig:shape-comparison}
\end{figure}

Let us now see how the positivity and monotonicity of $Z$ help us
study the sign of the gradient moment.
Since $w_x=Z\sin x$ and $\int_0^\pi wf=0$, $w$ is then
strictly increasing and has a unique root
$c\in(0,\pi)$. By~\eqref{eq:Z-definition}, we may write the gradient moment as
\begin{equation}\label{eq:Z-moment-factorization}
\int_0^\pi wz^2f=\int_0^\pi wZ^2\sin^2x\,f.
\end{equation}
By the monotonicity of $Z$, replacing $Z^2$ in~\eqref{eq:Z-moment-factorization} by the constant
$Z(c)^2$ gives higher weight on the negative contributions of $w$ and
a lower weight on the positive contributions. We thus get
\begin{equation}\label{eq:gradient-reduction}
\int_0^\pi wz^2f\geq Z(c)^2\int_0^\pi w\sin^2x\,f.
\end{equation}
Hence, since $Z(c)>0$, to prove the positivity of the gradient moment, it is enough to show
\begin{equation}\label{eq:weighted-w-target}
\int_0^\pi w(x)\sin^2x\,f(x)\,dx>0.
\end{equation}

The inequality~\eqref{eq:weighted-w-target} is a much more tractable goal, because unlike $z^2$, the weight $\sin^2x$ is an explicit symmetric function. In Section~\ref{subsec:gradient}, we prove \eqref{eq:weighted-w-target} through a symmetrization argument that exploits one further monotonicity property analogous to the ones discovered through~\eqref{eq:quotients-small-zeta}.

Next, we observe that for the cubic moment $\int_\T w^3f$, neither $w$
nor $f$ are known explicitly. Our strategy is then to do a change of variables
in this integral to replace $w$ by an explicit function, while
transforming $f$ into a new density $\widetilde f$. It will then suffice to
find a condition on the single unknown function $\widetilde f$ under which the integral is nonnegative.

At $\zeta=0$, we have $w=-\cos x/(1+\lambda)$, which, together with the favorable monotonicity observations made so far, strongly suggests
choosing the explicit function to be a cosine. Since $w$ is strictly increasing on $[0,\pi]$
for $\zeta>0$, we view it as a spatial deformation of this cosine
profile by writing
\begin{equation}\label{eq:w-cosine-overview}
w(x)=b\bigl(\bar c-\cos\theta(x)\bigr).
\end{equation}
Here $b>0$ is a constant, $\theta$ is an increasing bijection of
$[0,\pi]$ onto itself, and $\bar c$ is determined uniquely by the condition
$\int_0^\pi wf=0$.

The probability density $\widetilde f$ is obtained from $2f(x)\,dx$
on $[0,\pi]$ by this change of variables, so that
$\widetilde f(\theta(x))=2f(x)/\theta'(x)$. By~\eqref{eq:w-cosine-overview},
\begin{equation}\label{eq:cubic-transform-overview}
\int_\T w^3f
=b^3\int_0^\pi(\bar c-\cos\theta)^3\widetilde f(\theta)\,d\theta.
\end{equation}
Since the zero-mean condition determines
$\bar c$ from $\widetilde f$, the sign is completely determined by $\widetilde f$. We the show in 
Lemma~\ref{lem:monotone-cosine-skewness} that
the integral is nonnegative as long as $\widetilde f$ is nonincreasing.

At $\zeta=0$, we have $\theta=x$ and $\widetilde f=1/\pi$.
We substitute~\eqref{eq:v-small-zeta} into~\eqref{eq:mu} to expand $f$.
To expand $\theta$, we integrate $z=w_x$ using~\eqref{eq:u-z-small-zeta}
and apply~\eqref{eq:w-cosine-overview}. This gives
\begin{equation}\label{eq:density-factors-small-zeta}
\begin{aligned}
2f(x)&=\frac1\pi\left(1+\frac{\zeta\cos x}{1+\lambda}\right)
+O(\zeta^2),\\
\frac1{\theta'(x)}&=1+\frac{\zeta\cos x}{(1+\lambda)(4+\lambda)}
+O(\zeta^2).
\end{aligned}
\end{equation}
Both factors in $\widetilde f(\theta(x))=2f(x)/\theta'(x)$ have decreasing
first-order corrections. This suggests that $\widetilde f$ is decreasing
for small positive $\zeta$ and, as before, numerics strongly suggest that the property persists for all $\zeta>0$. Indeed,
in Section~\ref{subsec:cubic} we prove that $\widetilde f$ is nonincreasing
for every $\zeta>0$, completing the proof that the cubic moment is nonnegative and hence the proof
of Theorem~\ref{thm:main}.

\section{Preliminaries}\label{sec:setup}

We use the rescaling introduced in Section~\ref{subsec:proof-overview},
\[
\lambda=\frac{2\beta}{\sigma^2},\qquad
\zeta=\frac{4\gamma}{\sigma^4},\qquad
v^\zeta=\frac{2}{\sigma^2}\phi^\gamma.
\]
We fix $\lambda>0$ and first establish existence, uniqueness,
and smooth dependence on $\zeta$ for solutions of~\eqref{eq:HJB},
\[
-v^\zeta_{xx}+\lambda v^\zeta+\tfrac12(v^\zeta_x)^2=-\zeta\cos x,
\qquad x\in\T.
\]
Let $C^{k,\alpha}(\T)$ be the Banach space of $C^{k,\alpha}$ functions on $\T$ (or equivalently $2\pi$-periodic functions on $\mathbb R$), endowed with the standard H\"older norm.

\begin{proposition}[Well-posedness and smooth dependence]\label{prop:smooth-zeta}
For every $\zeta\geq 0$, \eqref{eq:HJB} has a unique solution
$v^\zeta\in C^\infty(\T)$. The solution is even with respect to both
$0$ and $\pi$, and every partial derivative $\partial_\zeta^j\partial_x^k v^\zeta$ is
even with respect to $0$ and $\pi$ when $k$ is even, and odd with respect to $0$ and $\pi$ when $k$
is odd. Moreover, for every $k\geq 0$ and $\alpha\in(0,1)$, the mapping
$\zeta\mapsto v^\zeta$ is $C^\infty$ as a map into
$C^{k,\alpha}(\T)$.
\end{proposition}
\begin{proof} Fix
$M>0$ and $\alpha\in(0,1)$, and let $\mathcal I_M\subset[0,M]$ be the set
of $\zeta$ for which~\eqref{eq:HJB} admits a solution in
$C^{2,\alpha}(\T)$. Since $v\equiv0$ solves~\eqref{eq:HJB} at
$\zeta=0$, the set $\mathcal I_M$ is nonempty.

We first show that \(\mathcal I_M\) is open. The argument below will also show that each solution extends to a smooth family of solutions for nearby values of \(\zeta\). Assume that $v$ is a solution to
\eqref{eq:HJB} with $\zeta=\zeta_0\in\mathcal I_M$, and define
$\mathcal F\colon C^{2,\alpha}(\T)\times\mathbb R\to C^{0,\alpha}(\T)$ by
\[
\mathcal F(v,\zeta):=-v_{xx}+\lambda v+\frac12 v_x^2+\zeta\cos x.
\]
Differentiating with respect to the $v$ variable, we get
\[
D_v\mathcal F(v,\zeta_0)h=-h_{xx}+v_xh_x+\lambda h.
\]
By the maximum principle, using that $\lambda>0$, the kernel is trivial.
The linearized operator $D_v\mathcal F(v,\zeta_0)$ is Fredholm of index
zero. Indeed, letting $L_0=-\partial_{xx}+1$, we have
\[
D_v\mathcal F(v,\zeta_0)L_0^{-1}=I+K,\qquad
K=(v_x\partial_x+\lambda-1)L_0^{-1},
\]
and,  by the compact
embedding $C^{1,\alpha}(\T)\hookrightarrow
C^{0,\alpha}(\T)$, $K$ is compact on $C^{0,\alpha}(\T)$. Thus $D_v\mathcal F(v,\zeta_0)$ is an
isomorphism by the Fredholm alternative~\cite[Thm.~5.3]{GT}. The
implicit function theorem then yields neighborhoods $\mathcal U$ of $v$ and $J$ of
$\zeta_0$, and a $C^\infty$ map
$\zeta\in J\mapsto \tilde v^\zeta\in\mathcal U$, such that
$\mathcal F(\tilde v^\zeta,\zeta)=0$. In particular,
$J\cap[0,M]\subset\mathcal I_M$, so $\mathcal I_M$ is open.

We show next that $\mathcal I_M$ is closed. Let
$\zeta_n\in\mathcal I_M$ with $\zeta_n\to\zeta\in[0,M]$, and let $v_n$
solve~\eqref{eq:HJB} with $\zeta=\zeta_n$. By the maximum principle,
\[
\|v_n\|_{L^\infty}\leq \frac M\lambda .
\]
Writing $u_n:=v_{n,x}$ and differentiating~\eqref{eq:HJB}, we have
\[
-u_{n,xx}+\lambda u_n+u_nu_{n,x}=\zeta_n\sin x,
\]
and maximum principle applied to $u_n$ then yields
$\|u_n\|_{L^\infty}\leq M/\lambda$. Returning to~\eqref{eq:HJB}, we deduce that
\begin{equation}\label{eq:vnxx}
v_{n,xx}=\lambda v_n+\frac12 u_n^2+\zeta_n\cos x
\end{equation}
is uniformly bounded. Hence $(v_n)$ is precompact in $C^1(\T)$. Along
a subsequence, $v_n\to v$ and $v_{n,x}\to v_x$ uniformly. Equation~\eqref{eq:vnxx}
then gives uniform convergence of $v_{n,xx}$, and
the limit $v\in C^2(\T)$ solves~\eqref{eq:HJB} at the parameter
$\zeta$. By Schauder estimates, we then upgrade $v$ to $C^{2,\alpha}(\T)$,
so $\zeta\in\mathcal I_M$. Thus
$\mathcal I_M=[0,M]$. Since $M>0$ was arbitrary, there exists a solution for every
$\zeta\geq0$.

Uniqueness then follows from the maximum principle. Indeed, if $v_1$ and $v_2$ solve
\eqref{eq:HJB} at for the same value of $\zeta$, then $\psi:=v_1-v_2$ satisfies
\[
-\psi_{xx}+\frac12(v_{1,x}+v_{2,x})\psi_x+\lambda\psi=0,
\]
and thus $\psi\equiv0$. By uniqueness, the smooth family $\tilde v^\zeta$
given by the implicit function theorem must coincide with $v^\zeta$.
Thus $v^\zeta$ depends smoothly on $\zeta$ in the
$C^{2,\alpha}(\T)$ norm. Solving~\eqref{eq:HJB} for $v_{xx}$ and bootstrapping then yields smooth dependence in the $C^{k,\alpha}(\T)$ norm for all $k\geq3$.

Since~\eqref{eq:HJB} is invariant under $x\mapsto-x$,
uniqueness gives $v^\zeta(-x)=v^\zeta(x)$. Combining this with
periodicity yields
$v^\zeta(\pi+y)=v^\zeta(-\pi-y)=v^\zeta(\pi-y)$, thus $v^\zeta$ is also
even with respect to $\pi$. The symmetry properties of the derivatives then follow by
differentiating the identities $v^\zeta(-x)=v^\zeta(x)$ and
$v^\zeta(\pi+y)=v^\zeta(\pi-y)$.

\end{proof}

Recall from~\eqref{eq:mu} and~\eqref{eq:A} the density $f^\zeta$
of $\mu^\zeta$ and its cosine moment $A(\zeta)$:
\[
f^\zeta(x)=\frac{e^{-v^\zeta(x)}}{\mathcal Z^\zeta},\qquad
\mathcal Z^\zeta=\int_\T e^{-v^\zeta(y)}\,dy,\qquad
A(\zeta)=\int_\T\cos x\,f^\zeta(x)\,dx.
\]
By~\eqref{eq:F-rescaling}, the fixed-point equation
$\gamma=F_\kappa(\gamma)$ takes the form
\[
\zeta=\frac{4\kappa}{\sigma^4}A(\zeta).
\]
By a first-order computation in $\zeta$, we now identify the critical threshold where bifurcation occurs.
Using Proposition~\ref{prop:smooth-zeta}, we may define
$w:=\partial_\zeta v^\zeta|_{\zeta=0}\in C^\infty(\T)$.
Given that $v^0\equiv0$, differentiating~\eqref{eq:HJB} with respect to $\zeta$, and setting
$\zeta=0$, we get $-w_{xx}+\lambda w=-\cos x$, and thus
\begin{equation}\label{eq:v-zeta-zero}
 w(x)=-\frac{\cos x}{\lambda+1}.
\end{equation}
Note that, since $A(0)=0$, this agrees at $\zeta=0$ with the definition
of $w$ in~\eqref{eq:w-z-intro}.
Substituting $f^\zeta=e^{-v^\zeta}/\mathcal Z^\zeta$ from~\eqref{eq:mu}
into~\eqref{eq:A}, we differentiate at $\zeta=0$.
Since $f^0\equiv1/(2\pi)$ and $\int_\T\cos x\,dx=0$,
the term appearing from differentiation of $1/\mathcal Z^\zeta$ is zero, giving
\[
A'(0)=-\int_\T \cos x\,w(x)\,\frac{dx}{2\pi}
=\frac{1}{2\pi(\lambda+1)}\int_\T\cos^2 x\,dx
=\frac{1}{2(\lambda+1)}.
\]
Hence, using ~\eqref{eq:F-rescaling} and the definition of $\kappa_c$
in~\eqref{eq:kappa-c}, we conclude
\begin{equation}\label{eq:F-properties}
F_\kappa(0)=0,\qquad F_\kappa'(0)=\frac{\kappa}{\kappa_c},\qquad |F_\kappa|\leq\kappa.
\end{equation}
To prove Theorem~\ref{thm:main}(i), it then remains to prove that, for every $\zeta>0$, $A''(\zeta)<0$.

\section{\texorpdfstring{The second derivative and shape estimates}{The second derivative and shape estimates}}\label{sec:prelim}

We fix $\lambda>0$ and write $v=v^\zeta$, $u=v_x$,
$f=f^\zeta$, and $\mathcal Z=\mathcal Z^\zeta$.

\subsection{\texorpdfstring{The second derivative}{The second derivative}}\label{subsec:second-derivative}

We first derive an identity for $A''$ that reduces the concavity proof to two integral inequalities.

\begin{lemma}[Second-derivative identity]\label{lem:second-deriv}
Let
\[
w:=\frac{A}{\lambda}+v_\zeta,\qquad z:=v_{x\zeta}.
\]
Then
\begin{equation}\label{eq:second-derivative}
\int_\T wf=0,\qquad
A''(\zeta)=-\lambda\int_\T w^3f-3\int_\T wz^2f .
\end{equation}
\end{lemma}
\begin{proof}
To prove that $\int_\T wf=0$, we differentiate~\eqref{eq:HJB} with respect to $\zeta$
and multiply by $f$. Since $f_x=-uf$, we obtain
\begin{equation}\label{eq:v-zeta-eq}
-(fv_{x\zeta})_x+\lambda fv_\zeta=-f\cos x.
\end{equation}
Integrating ~\eqref{eq:v-zeta-eq} over $\T$ gives
$\lambda\int_\T v_\zeta f=-A$. Thus $w=A/\lambda+v_\zeta$ satisfies
$\int_\T wf=0$, as wanted.

We also differentiate $f=e^{-v^\zeta}/\mathcal Z^\zeta$ in~\eqref{eq:mu}.
The normalizing constant $\mathcal Z^\zeta$ satisfies
$\mathcal Z^\zeta_\zeta/\mathcal Z^\zeta=-\int_\T v_\zeta f=A/\lambda$,
so the identity $f_\zeta=-v_\zeta f-(\mathcal Z^\zeta_\zeta/\mathcal Z^\zeta)f$
can be written as
\begin{equation}\label{eq:f-zeta}
f_\zeta=-wf.
\end{equation}

We next write~\eqref{eq:v-zeta-eq} in terms of $w=A/\lambda+v_\zeta$:
\begin{equation}\label{eq:w-energy-eq}
-(fw_x)_x+\lambda fw=(A-\cos x)f .
\end{equation}
We differentiate $A=\int_\T\cos x\,f$ and apply~\eqref{eq:f-zeta}
to obtain
\[
A'=-\int_\T \cos x\,wf=\int_\T (A-\cos x)wf,
\]
where we used $\int_\T wf=0$ for the second equality.
We then multiply~\eqref{eq:w-energy-eq} by $w$ and integrate by parts. Recalling that $w_x=z$, we deduce that
\begin{equation}\label{eq:A-prime-energy}
A'=\int_\T z^2f+\lambda\int_\T w^2f .
\end{equation}

We set $b:=w_\zeta=A'/\lambda+v_{\zeta\zeta}$, so that $b_x=z_\zeta$.
We differentiate~\eqref{eq:HJB} twice with respect to $\zeta$ and multiply by $f$. Using that $f_x=-uf$, the equation becomes
\begin{equation}\label{eq:b-eq}
-(fb_x)_x+\lambda fb=A'f-z^2f .
\end{equation}
Multiplying~\eqref{eq:b-eq} by $w$ and integrating by parts once more, using $A'\int_\T wf=0$, we get
\begin{equation}\label{eq:b-tested}
\int_\T zz_\zeta f+\lambda\int_\T ww_\zeta f
=-\int_\T wz^2f .
\end{equation}
Finally, we differentiate~\eqref{eq:A-prime-energy} with respect to
$\zeta$ and use \eqref{eq:f-zeta}:
\[
A''=2\int_\T zz_\zeta f+2\lambda\int_\T ww_\zeta f
-\int_\T wz^2f-\lambda\int_\T w^3f.
\]
By~\eqref{eq:b-tested} the first two terms on the right hand side are equal to
$-2\int_\T wz^2f$, completing the proof of \eqref{eq:second-derivative}.
\end{proof}

Since $\lambda>0$, formula~\eqref{eq:second-derivative}
reduces the proof of concavity to the moment inequalities~\eqref{eq:moment-signs},
\[
\int_\T wz^2f>0,\qquad \int_\T w^3f\geq0.
\]

\subsection{\texorpdfstring{Shape estimates}{Shape estimates}}\label{subsec:shape}

We fix $\zeta>0$ and recall the normalized quantities defined
in~\eqref{eq:Z-definition}:
\[
U=\frac{u}{\sin x},\qquad Z=\frac{z}{\sin x}=\partial_\zeta U,
\qquad 0<x<\pi.
\]
In Section~\ref{subsec:proof-overview}, we explained why $U$ and $Z$
should be positive and nondecreasing and why the the monotonicity of $Z$ will be useful. In this section, we will prove these properties
for every $\zeta>0$.

Moreover, an expansion of $fZ$ near $\zeta=0$, similar to those for $U$ and $Z$
in~\eqref{eq:quotients-small-zeta}, suggests that $fZ$ is nonincreasing
for small positive $\zeta$. We will prove this property
for every $\zeta>0$ and later use it in the symmetrization argument that
closes the proof of the gradient-moment inequality in
Section~\ref{subsec:gradient}.
Since $f_x=-uf$, the inequality $(fZ)_x\leq0$ is equivalent to $Z_x\leq uZ$.
Altogether, the estimates to be proved in this section are therefore
\begin{equation}\label{eq:Z-bounds-overview}
U>0,\qquad U_x\geq0,\qquad Z>0,\qquad
0\leq Z_x\leq uZ\qquad\text{on }(0,\pi).
\end{equation}

To study $Z$, we differentiate~\eqref{eq:HJB} in $x$ and $\zeta$ and substitute
$z=Z\sin x$ to obtain the following equation for $Z$:
\begin{equation}\label{eq:Z-shape}
-Z_{xx}+(u-2\cot x)Z_x+\mathcal B Z=1,
\qquad \mathcal B:=1+\lambda+u_x+u\cot x.
\end{equation}
Differentiating~\eqref{eq:Z-shape} we get an elliptic equation for $Z_x$, with inhomogeneous term
$-\mathcal B_xZ$. Since we want to prove that $Z$ is positive
and nondecreasing, this suggests first establishing that $\mathcal B_x\leq0$.
We prove now prove this, together with the basic geometric properties of $u$.

\begin{lemma}[Shape of $u$]\label{lem:u-shape}
On $(0,\pi)$,
\begin{equation}\label{eq:shape}
u>0,\qquad u_{xx}\leq 0,\qquad
 U_x\geq 0.
\end{equation}
Moreover, letting $\mathcal B$ be as in~\eqref{eq:Z-shape}, we have $\mathcal B_x\leq 0$ on
$(0,\pi)$.
\end{lemma}
\begin{proof}
Since $v$ is even with respect to $0$ and $\pi$ by Proposition~\ref{prop:smooth-zeta},
$u=v_x$ is odd about both endpoints, and in particular $u(0)=u(\pi)=0$.
Differentiating~\eqref{eq:HJB} with respect to $x$, we get
\begin{equation}\label{eq:u-eq}
-u_{xx}+\lambda u+uu_x=\zeta\sin x.
\end{equation}

We first show that $u>0$ on $(0,\pi)$. We write~\eqref{eq:u-eq} as
$Lu=\zeta\sin x$, where $L$ is the linear operator
$Lw:=-w_{xx}+uw_x+\lambda w$. Since $u(0)=u(\pi)=0$ and $L(0)=0<\zeta\sin x$ on $(0,\pi)$, by the strong maximum principle we deduce $u>0$ on $(0,\pi)$.

We next prove that $u_{xx}\leq 0$ and $U_x\geq 0$.
Differentiating~\eqref{eq:u-eq} twice shows that
$u_{xx}$ satisfies an elliptic equation whose zero-order coefficient $\lambda+3u_x$ has an undetermined sign, which prevents applying the maximum principle directly.
Instead, we prove the two inequalities simultaneously using an auxiliary first-order system. We set
\begin{equation}\label{eq:Kdef}
\mathcal K:=\frac{\zeta}{U}-\lambda-u_x.
\end{equation}
By~\eqref{eq:u-eq}, we have $u_{xx}=-u\mathcal K$, and since $u>0$,
the inequality $u_{xx}\leq 0$ is reduced to $\mathcal K\geq 0$. We also define
\begin{equation}\label{eq:Wdef}
\mathcal W:=\sin x\,u_x-\cos x\,u=\sin^2x\,U_x.
\end{equation}
The function $\mathcal W$ is the Wronskian of $u$ and the forcing term
$\sin x$, and $\mathcal W\geq 0$ is equivalent to $U_x\geq 0$.
We differentiate~\eqref{eq:Wdef} and~\eqref{eq:Kdef} and substitute
$u_{xx}=-u\mathcal K$ in both expressions. This gives
\begin{subequations}\label{eq:WK}
\begin{align}
\mathcal W_x &= \sin x\,u\,(1-\mathcal K), \label{eq:Wx}\\
\mathcal K_x &= u\mathcal K-\frac{\zeta\mathcal W}{u^2} . \label{eq:Kx}
\end{align}
\end{subequations}

Equation~\eqref{eq:Wx} is the usual identity for the derivative of
a Wronskian (see, e.g.,~\cite{Hartman}). The normalized curvature of $u$
is $-u_{xx}/u=\mathcal K$, while that of the forcing term $\sin x$ is
$-(\sin x)_{xx}/\sin x=1$. Thus~\eqref{eq:Wx} compares the two:
$\mathcal W$ increases where $\mathcal K<1$ and decreases where
$\mathcal K>1$. In deriving~\eqref{eq:Kx}, we also used
$(1/U)_x=-\mathcal W/u^2$.

Suppose that $\mathcal W$ takes a negative value. Since
$\mathcal W(0)=\mathcal W(\pi)=0$ by~\eqref{eq:Wdef}, it must attain a
negative minimum at some $x_0\in(0,\pi)$. At this point,
$\mathcal W_x(x_0)=0$, and~\eqref{eq:Wx} implies that $\mathcal K(x_0)=1$.
From~\eqref{eq:Kx}, we then get
$\mathcal K_x(x_0)=u(x_0)-\zeta\mathcal W(x_0)/u(x_0)^2>0$.
Differentiating~\eqref{eq:Wx} once more and evaluating at $x_0$, we obtain
\[
\mathcal W_{xx}(x_0)=-\sin x_0\,u(x_0)\,\mathcal K_x(x_0)<0,
\]
which is impossible at a minimum. We conclude that $\mathcal W\geq 0$
on $(0,\pi)$.

Consider now a zero of $\mathcal K$ in $(0,\pi)$.
By~\eqref{eq:Kx}, $\mathcal K_x=-\zeta\mathcal W/u^2\leq 0$.
Equality would imply $\mathcal W=0$, making this point a local minimum
of $\mathcal W$, since $\mathcal W\geq 0$. We would then have
$\mathcal W_x=0$, whereas~\eqref{eq:Wx} gives $\mathcal W_x=\sin x\,u>0$.
Hence $\mathcal K_x<0$ at every zero of $\mathcal K$ in $(0,\pi)$.

We next show that $\mathcal K$ has a positive limit as $x\to\pi^-$.
Let $\xi:=\pi-x$. By oddness of $u$ about $\pi$, we have
$u=a\xi+b\xi^3+o(\xi^3)$, where $a:=-u_x(\pi)$.
We substitute this expansion into~\eqref{eq:u-eq} and compare
coefficients of $\xi$, obtaining
\begin{equation}\label{eq:ab}
6b=\lambda a-a^2-\zeta .
\end{equation}
Since $u>0$ on $(0,\pi)$ and $u(\pi)=0$, we have $a\geq 0$.
If $a=0$, then~\eqref{eq:ab} would give $b=-\zeta/6<0$, and the expansion
would imply $u\sim-\zeta\xi^3/6<0$ for small $\xi>0$, contradicting
$u>0$ on $(0,\pi)$. Thus $a>0$, and L'H\^opital's rule gives
$1/U=\sin x/u\to 1/a$ as $x\to\pi^-$. From~\eqref{eq:Kdef}, we obtain
\[
\mathcal K(\pi^-)=\frac{\zeta}{a}+a-\lambda=\frac{a^2-\lambda a+\zeta}{a}.
\]
We also substitute the expansion of $u$ into~\eqref{eq:Wdef} and
use~\eqref{eq:ab} to find
\[
\mathcal W=\frac{a^2-(\lambda+1)a+\zeta}{3}\,\xi^3+o(\xi^3).
\]
The coefficient of $\xi^3$ must be nonnegative because
$\mathcal W\geq 0$. We thus have
$a^2-(\lambda+1)a+\zeta\geq 0$, namely
$\mathcal K(\pi^-)\geq 1>0$.
In particular, $\mathcal K>0$ for $x<\pi$ sufficiently close to $\pi$.
Suppose that $\mathcal K(x_*)<0$ for some $x_*\in(0,\pi)$.
By continuity, we can find $x_0\in(x_*,\pi)$ such that
$\mathcal K(x_0)=0$ and $\mathcal K<0$ on $[x_*,x_0)$.
Then $\mathcal K_x(x_0)\geq0$, contradicting the fact that $\mathcal K_x<0$ at every zero.
We conclude that $\mathcal K\geq0$ on $(0,\pi)$, and hence $u_{xx}\leq0$.

Since $u$ is concave and $u(0)=u(\pi)=0$, we have $u_x\leq u/x$
at points where $u_x>0$ and $-u_x\leq u/(\pi-x)$ at points where $u_x<0$.
Using $\sin x\leq\min(x,\pi-x)$ on $(0,\pi)$, we obtain the chord bound
\begin{equation}\label{eq:chord}
|u_x|\leq U=\frac{u}{\sin x} .
\end{equation}
We then differentiate $\mathcal B$:
\begin{equation}\label{eq:B-derivative}
\mathcal B_x=u_{xx}+\frac{u_x\sin x\cos x-u}{\sin^2x}.
\end{equation}
We already know that $u_{xx}\leq 0$, and~\eqref{eq:chord} gives
$u_x\sin x\cos x\leq|u_x|\sin x\leq u$, since $|\cos x|\leq 1$.
The numerator in the right hand side of ~\eqref{eq:B-derivative} is hence
nonpositive, and we
conclude that $\mathcal B_x\leq 0$.
\end{proof}

We now use the estimates for $u$ to prove the announced properties of $Z$. 

\begin{lemma}[Bounds for $Z$]\label{lem:Z-bounds}
Let
\[
z:=v_{x\zeta},\qquad Z:=\frac{z}{\sin x}.
\]
Then $z>0$ and $Z>0$ on $(0,\pi)$, and
\[
0\leq Z_x\leq uZ\qquad\text{on }(0,\pi).
\]
The upper bound is equivalent to $(fZ)_x\leq 0$.
\end{lemma}
\begin{proof}
We first show that $z>0$ on $(0,\pi)$. Differentiating~\eqref{eq:HJB}
with respect to $\zeta$ and $x$ yields
\begin{equation}\label{eq:z-eq}
-z_{xx}+\lambda z+(uz)_x=\sin x.
\end{equation}
By Proposition~\ref{prop:smooth-zeta}, $z=v_{x\zeta}$ is odd with respect to
$0$ and $\pi$, so $z(0)=z(\pi)=0$.

Note that the zeroth-order coefficient in \eqref{eq:z-eq} is $\lambda+u_x$ and may not be positive, so to get an amenable elliptic equation we
substitute $z=e^v y$ in~\eqref{eq:z-eq}, obtaining
\begin{equation}\label{eq:y-eq}
-y_{xx}-uy_x+\lambda y=e^{-v}\sin x>0 ,
\end{equation}
with $y(0)=y(\pi)=0$. If $y$ were nonpositive at some point in
$(0,\pi)$, it would have an interior minimum $x_0$ with $y(x_0)\leq0$.
Then $y_x(x_0)=0$ and $y_{xx}(x_0)\geq0$, so~\eqref{eq:y-eq} gives
$\lambda y(x_0)=y_{xx}(x_0)+e^{-v(x_0)}\sin x_0>0$, which is a contradiction. This proves that $y>0$ on $(0,\pi)$. Since $z=e^v y$ and $Z=z/\sin x$,
we infer that $z$ and $Z$ are also positive on $(0,\pi)$.
Both $z$ and $\sin x$ are odd with respect to $0$ and $\pi$,
so $Z=z/\sin x$ is smooth and even with respect to both $0$ and $\pi$.
In particular, $Z_x(0)=Z_x(\pi)=0$.

To show that $Z_x\geq0$, we differentiate~\eqref{eq:Z-shape}
and make the substitution $Z_x=e^v q$. Using $u_x-u\cot x=\sin x\,U_x$, we obtain
\begin{equation}\label{eq:Z-derivative}
-q_{xx}-(u+2\cot x)q_x
+\bigl(\lambda+1+\sin x\,U_x+2\csc^2x\bigr)q
=-e^{-v}\mathcal B_xZ.
\end{equation}
By Lemma~\ref{lem:u-shape}, $U_x\geq0$ and $\mathcal B_x\leq0$.
Since $Z>0$, the right-hand side of~\eqref{eq:Z-derivative} is
nonnegative, and the zeroth-order coefficient is positive.

Note also that $q(0)=q(\pi)=0$. If $q$ had a negative minimum at
$x_0\in(0,\pi)$, then $q_x(x_0)=0$ and $q_{xx}(x_0)\geq0$, so the
left-hand side of~\eqref{eq:Z-derivative} would be strictly negative. This is a contradiction, so we deduce that $q\geq0$. Recalling that $Z_x=e^v q$,
we conclude that
\begin{equation}\label{eq:Z-mono}
Z_x\geq 0.
\end{equation}

To prove $Z_x\leq uZ$, we set
\[
\Sigma:=uZ-Z_x=-\frac{(fZ)_x}{f}.
\]
We need to show that $\Sigma\geq0$. Differentiating $\Sigma$ and
using~\eqref{eq:Z-shape} yields
\begin{equation}\label{eq:Sigma}
\Sigma_x=1-A_0Z-2\cot x\,\Sigma,\qquad A_0:=1+\lambda-u\cot x .
\end{equation}
Differentiating $A_0$ and using the chord bound~\eqref{eq:chord}, we get
\[
(A_0)_x=\frac{u-u_x\sin x\cos x}{\sin^2x}\geq 0.
\]
Since $u$ and $Z_x$ vanish at $0$ and $\pi$, we have
$\Sigma(0)=\Sigma(\pi)=0$. Suppose, by contradiction, that $\Sigma$ attains a negative value,
and let $x_0\in(0,\pi)$ be a point where it attains its minimum.
We first check that $A_0(x_0)>0$. If $x_0\in(\pi/2,\pi)$, then
$u(x_0)\cot x_0\leq0$ and $A_0(x_0)\geq1+\lambda>0$.
If $x_0\in(0,\pi/2]$, we use $\Sigma_x(x_0)=0$ in~\eqref{eq:Sigma}
to obtain
$A_0(x_0)Z(x_0)=1-2\cot x_0\,\Sigma(x_0)\geq1>0$,
because $\cot x_0\geq0$ and $\Sigma(x_0)<0$. Since $Z(x_0)>0$,
we again get $A_0(x_0)>0$. Moreover,
$Z_x(x_0)=u(x_0)Z(x_0)-\Sigma(x_0)>0$, and we deduce that
$(A_0Z)_x(x_0)=\bigl((A_0)_xZ+A_0Z_x\bigr)(x_0)>0$.
We differentiate~\eqref{eq:Sigma} and use $\Sigma_x(x_0)=0$ to obtain
\[
\Sigma_{xx}(x_0)=-(A_0Z)_x(x_0)+2\csc^2x_0\,\Sigma(x_0)<0 ,
\]
which contradicts $\Sigma_{xx}(x_0)\geq0$.
Thus $\Sigma\geq0$, and we have proved
\begin{equation}\label{eq:Z-upper}
Z_x\leq uZ.
\end{equation}

Finally, since $f_x=-uf$, the bound~\eqref{eq:Z-upper} implies
\[
(fZ)_x=f(Z_x-uZ)\leq0.
\qedhere
\]
\end{proof}

\section{Proof of Theorem~\ref{thm:main}}\label{sec:proof}

We fix $\zeta>0$ and use the estimates of
Lemmas~\ref{lem:u-shape} and~\ref{lem:Z-bounds}
to prove the moment inequalities~\eqref{eq:moment-signs}.

\subsection{\texorpdfstring{The gradient moment}{The gradient moment}}\label{subsec:gradient}

We first prove a more general form of~\eqref{eq:weighted-w-target},
in which $\sin^2x$ is replaced by a nonconstant $C^1$ weight $\Omega$ that is symmetric
with respect to $\pi/2$ and nondecreasing on $[0,\pi/2]$.

Using the symmetry of $\Omega$, we split $\int_0^\pi w\Omega f$
at $\pi/2$ and make the change of variables $x\mapsto\pi-x$ in the
second half to obtain
\begin{equation}\label{eq:omega-fold}
\int_0^\pi w(x)\Omega(x)f(x)\,dx
=\int_0^{\pi/2}\Omega(x)\bigl[w(x)f(x)+w(\pi-x)f(\pi-x)\bigr]\,dx.
\end{equation}
Since $f$ is not symmetric, in order to preserve $f(x)\,dx$ as the
integration measure in~\eqref{eq:omega-fold}, we define the \emph{folded profile} by
\begin{equation}\label{eq:folded-profile}
\tilde w(x):=w(x)+\frac{f(\pi-x)}{f(x)}\,w(\pi-x),
\qquad 0\leq x\leq\pi/2,
\end{equation}
so that right-hand side of~\eqref{eq:omega-fold} is then
$\int_0^{\pi/2}\Omega\tilde w f$.

By evenness and~\eqref{eq:second-derivative},
$\int_0^{\pi/2}\tilde w f=\int_0^\pi wf=0$.
In the proof below, we use the monotonicity of $fZ$ to show that
$\tilde w$ is strictly increasing.
A nonconstant weight $\Omega$ that is nondecreasing on $[0,\pi/2]$
then gives greater weight to the positive part of $\tilde w$, which yields the lemma.

\begin{lemma}[Positivity for symmetric weights]\label{lem:central-positivity}
Assume $\zeta>0$. Let $\Omega\in C^1([0,\pi])$ be non-constant and
satisfy
\[
\Omega(\pi-x)=\Omega(x),\qquad \Omega'(x)\geq 0\quad\text{for }0<x<\pi/2.
\]
Then, with $w$ as in Lemma~\ref{lem:second-deriv},
\begin{equation}\label{eq:omega-positive}
\int_0^\pi w(x)\,\Omega(x)\,f(x)\,dx> 0 .
\end{equation}
\end{lemma}
\begin{proof}
Set $\bar x:=\pi-x$. By Lemma~\ref{lem:Z-bounds},
$w_x=z>0$, so $w$ is strictly increasing on $(0,\pi)$. Since, by Proposition~\ref{prop:smooth-zeta}, $wf$ is
even with respect to $0$,
Lemma~\ref{lem:second-deriv} gives $\int_0^\pi wf=0$.

Set
\[
\rho(x):=\frac{f(\bar x)}{f(x)},\qquad 0\leq x\leq \pi/2.
\]
Since $f_x=-uf$, we have $\rho'(x)=\rho(x)(u(x)+u(\bar x))$.
By Lemma~\ref{lem:u-shape}, we have $u>0$ on $(0,\pi)$, and thus $\rho'(x)>0$
for $0<x<\pi/2$. In particular, $\rho$ is strictly increasing on
$[0,\pi/2]$, with $\rho(\pi/2)=1$.
For any integrable function $g$, substituting $y=\bar x$ gives
\begin{equation}\label{eq:reflect-identity}
\int_0^{\pi/2} \rho(x)f(x)\,g(\bar x)\,dx
=\int_{\pi/2}^\pi g(y)f(y)\,dy.
\end{equation}
By Lemma~\ref{lem:Z-bounds}, $fZ$ is
nonincreasing, so
\[
f(x)Z(x)\geq f(\bar x)Z(\bar x),\qquad 0\leq x\leq\pi/2.
\]
We multiply by $\sin x$ and use $z=Z\sin x$ and $\sin x=\sin\bar x$
to obtain
\begin{equation}\label{eq:fz-reflect}
f(x)z(x)\geq f(\bar x)z(\bar x),\quad\text{i.e.}\quad
z(x)\geq\rho(x)z(\bar x),\qquad 0\leq x\leq \pi/2.
\end{equation}

We write the folded profile~\eqref{eq:folded-profile} as
\[
\tilde w(x)=w(x)+\rho(x)\,w(\bar x),\qquad 0\leq x\leq \pi/2.
\]
Applying~\eqref{eq:reflect-identity} with $g=w$, we obtain
\begin{equation}\label{eq:wt-zero-mean}
\int_0^{\pi/2} \tilde w(x)f(x)\,dx=\int_0^\pi wf=0 .
\end{equation}

We next show that $\tilde w$ is strictly increasing.
Differentiating $\tilde w$, we obtain
\begin{equation}\label{eq:wt-mono}
\tilde w'(x)=\bigl[z(x)-\rho(x)z(\bar x)\bigr]+\rho'(x)\,w(\bar x)
\qquad\text{on }(0,\pi/2).
\end{equation}
The term $z(x)-\rho(x)z(\bar x)$ is nonnegative by~\eqref{eq:fz-reflect}.
Since $\rho'>0$ on $(0,\pi/2)$ and $w(\bar x)\geq w(\pi/2)$
by monotonicity of $w$, it suffices to prove that $w(\pi/2)>0$.

Fix $x\in[0,\pi/2)$.
For $s\in[x,\pi/2]$, \eqref{eq:fz-reflect} and the monotonicity of $\rho$ give
\[
z(s)\geq \rho(s)z(\bar s)\geq \rho(x)z(\bar s).
\]
The second inequality is strict for $x<s<\pi/2$.
Using $w_x=z$ and $(d/ds)w(\bar s)=-z(\bar s)$, we integrate
from $x$ to $\pi/2$ and obtain
$w(\pi/2)-w(x)>\rho(x)\bigl(w(\bar x)-w(\pi/2)\bigr)$.
Rearranging this, we get
\begin{equation}\label{eq:wt-below-mid}
\tilde w(x)<(1+\rho(x))\,w(\pi/2),\qquad 0\leq x<\pi/2.
\end{equation}
Integrating~\eqref{eq:wt-below-mid} against $f(x)\,dx$ on $[0,\pi/2]$,
and using~\eqref{eq:wt-zero-mean} and~\eqref{eq:reflect-identity}
with $g=1$, we obtain $0<w(\pi/2)\int_0^\pi f$.
Since $\int_0^\pi f>0$, we conclude that $w(\pi/2)>0$.

We then infer from~\eqref{eq:wt-mono} that $\tilde w'>0$
on $(0,\pi/2)$.

Since $\tilde w$ is strictly increasing, $f>0$, and
$\int_0^{\pi/2}\tilde w f=0$ by~\eqref{eq:wt-zero-mean},
$\tilde w$ has a unique root $c\in(0,\pi/2)$.
We then use~\eqref{eq:omega-fold} and~\eqref{eq:wt-zero-mean} to write
\[
\int_0^\pi w(x)\Omega(x)f(x)\,dx
=\int_0^{\pi/2}\bigl(\Omega(x)-\Omega(c)\bigr)\tilde w(x)f(x)\,dx.
\]
By monotonicity, both $\Omega(x)-\Omega(c)$ and $\tilde w(x)$ are
nonpositive for $x<c$ and nonnegative for $x>c$.
Since $\Omega$ is continuous and nonconstant on $[0,\pi/2]$,
the product $(\Omega(x)-\Omega(c)\bigr)\tilde w(x)$ is positive on some interval. Since $f>0$, this completes the proof of~\eqref{eq:omega-positive}.
\end{proof}

We now use Lemma~\ref{lem:central-positivity}, together with the
properties of $Z$ supplied by Lemma~\ref{lem:Z-bounds},
to prove positivity for the gradient moment.

\begin{proposition}[Sign of the gradient moment]\label{prop:gradient}
For all $\zeta>0$, we have
\[
\int_\T w\,z^2 f> 0.
\]
\end{proposition}
\begin{proof}
Let $Z$ be the function from Lemma~\ref{lem:Z-bounds}. Since the integrand is even, it is enough to prove $\int_0^\pi wz^2f>0$. Moreover, since $w_x=z>0$ by Lemma~\ref{lem:Z-bounds}, $w$ is strictly
increasing on $[0,\pi]$. Given that $w$ also has mean zero, $w$ thus has a unique zero $c\in(0,\pi)$.

Since $Z$ is positive and nondecreasing, the functions $w(x)$ and
$Z(x)^2-Z(c)^2$ are both nonpositive on $(0,c)$ and both
nonnegative on $(c,\pi)$. Hence
\[
w(x)\bigl(Z(x)^2-Z(c)^2\bigr)\geq0,\qquad 0<x<\pi.
\]
We multiply by $\sin^2x\,f(x)$ and integrate over $[0,\pi]$.
Since $z=Z\sin x$ and $Z(c)>0$, we have
\[
\int_0^\pi wz^2f
=\int_0^\pi wZ^2\sin^2x\,f
\geq Z(c)^2\int_0^\pi w\sin^2x\,f>0,
\]
with the strict inequality following by applying
Lemma~\ref{lem:central-positivity} to $\Omega(x)=\sin^2x$.
Since the integrand is even, $\int_\T wz^2 f=2\int_0^\pi wz^2 f>0$.
\end{proof}

\subsection{\texorpdfstring{The cubic moment}{The cubic moment}}\label{subsec:cubic}

It remains to prove that $\int_\T w^3f\geq0$. At $\zeta=0$,
we have $w=-\cos x/(1+\lambda)$ and $f=1/(2\pi)$, so the cubic
moment is zero. For $\zeta>0$, as discussed in the introduction, we use the cosine change of variables
 \eqref{eq:w-cosine-overview} and prove that
the transformed density $\tilde f$ is nonincreasing.

Set
\[
\Delta_0(x):=w(x)-w(0),\qquad
\Delta_\pi(x):=w(\pi)-w(x),\qquad b:=\frac{w(\pi)-w(0)}{2}.
\]
By Lemma~\ref{lem:Z-bounds}, $w_x=z>0$ on $(0,\pi)$, and
so $b>0$. We define an increasing bijection
$\theta\colon[0,\pi]\to[0,\pi]$ by
\begin{equation}\label{eq:theta-change}
w(x)=w(0)+b(1-\cos\theta(x)).
\end{equation}
Note that at $\zeta=0$, we simply have $\theta(x)=x$.

The density of $2f(x)\,dx$ in the variable $\theta$ is
defined on $[0,\pi]$ by
\begin{equation}\label{eq:cosine-density}
\widetilde f(\theta(x))=\frac{2f(x)}{\theta'(x)}.
\end{equation}
Since $f$ is even and $\int_\T f=1$, we have
$\int_0^\pi\widetilde f(\theta)\,d\theta=2\int_0^\pi f(x)\,dx=1$,
so $\widetilde f$ is a probability density.

We integrate \eqref{eq:theta-change} against $2f(x)\,dx$ and use the evenness
of $wf$ and~\eqref{eq:second-derivative} to obtain
\begin{equation}\label{eq:affine-C}
w(x)=b\bigl(\bar c-\cos\theta(x)\bigr),\qquad
\bar c:=\int_0^\pi\cos\theta\,\widetilde f(\theta)\,d\theta.
\end{equation}
We substitute~\eqref{eq:affine-C} into $\int_\T w^3f$
and change variables from $x$ to $\theta$ using~\eqref{eq:cosine-density}.
Since $w$ and $f$ are even, we obtain
\begin{equation}\label{eq:cubic-moment-transform}
\int_\T w^3f\,dx
=b^3\int_0^\pi(\bar c-\cos\theta)^3\widetilde f(\theta)\,d\theta.
\end{equation}
We first show that the integral on the right-hand side
of~\eqref{eq:cubic-moment-transform} is nonnegative whenever $\widetilde f$
is nonincreasing on $[0,\pi]$.
We note that the inequality in the following Lemma is sharp, since equality is attained for the uniform density $\widetilde f\equiv1/\pi$.

\begin{lemma}[Monotone cosine skewness]\label{lem:monotone-cosine-skewness}
Let $\widetilde f\geq0$ be a nonincreasing probability density on $[0,\pi]$ and set
\[
\bar c:=\int_0^\pi \cos\theta\,\widetilde f(\theta)\,d\theta .
\]
Then
\[
\int_0^\pi(\bar c-\cos\theta)^3\widetilde f(\theta)\,d\theta\geq 0.
\]
\end{lemma}
\begin{proof}
Given that $\widetilde f$ is a probability density and $\cos\theta<1$ for
$0<\theta\leq\pi$, it follows that $\bar c<1$.
The monotonicity of $\widetilde f$ also implies that $\widetilde f(\theta)\geq\widetilde f(\pi-\theta)$
on $[0,\pi/2]$. Using the fact that $\cos(\pi-\theta)=-\cos\theta$, we then get
\[
\bar c=\int_0^{\pi/2}\cos\theta\,\bigl(\widetilde f(\theta)-\widetilde f(\pi-\theta)\bigr)\,d\theta\geq 0.
\]
By definition of $\bar c$,
\begin{equation}\label{eq:zero-mean-cos}
\int_0^\pi(\bar c-\cos\theta)\widetilde f(\theta)\,d\theta=0.
\end{equation}

Let $\theta_0:=\arccos\bar c$. On $[0,\theta_0]$, $\cos\theta$ decreases
from $1$ to $\bar c$. We set $\tau:=\cos\theta-\bar c$, so that
$\tau\in[0,1-\bar c]$ and
$d\tau=-\sin\theta\,d\theta=-\sqrt{1-(\bar c+\tau)^2}\,d\theta$.
The density in the variable $\tau$ is
\[
\alpha_+(\tau):=\frac{\widetilde f(\arccos(\bar c+\tau))}{\sqrt{1-(\bar c+\tau)^2}},\qquad\tau\in[0,1-\bar c].
\]
On the interval $[\theta_0,\pi]$, $\cos\theta$ decreases from $\bar c$ to $-1$.
Setting $\tau:=\bar c-\cos\theta$, we have $\tau\in[0,1+\bar c]$,
and we obtain the density
\[
\alpha_-(\tau):=\frac{\widetilde f(\arccos(\bar c-\tau))}{\sqrt{1-(\bar c-\tau)^2}},\qquad\tau\in[0,1+\bar c].
\]
Since $\bar c\geq0$, both densities are defined on $[0,1-\bar c]$.
For $\tau$ in this interval, $\bar c+\tau\geq\bar c-\tau$, and since
$\arccos$ is decreasing and $\widetilde f$ is nonincreasing, we have
$\widetilde f(\arccos(\bar c+\tau))\geq\widetilde f(\arccos(\bar c-\tau))$.
Also, $\bar c\geq0$ implies $(\bar c+\tau)^2\geq(\bar c-\tau)^2$, so
\[
\sqrt{1-(\bar c+\tau)^2}\leq \sqrt{1-(\bar c-\tau)^2}.
\]
Thus $\alpha_+\geq\alpha_-$ on $[0,1-\bar c]$.

After these changes of variables, \eqref{eq:zero-mean-cos} becomes
\begin{equation}\label{eq:zm-tau}
\int_0^{1-\bar c}\tau\,(\alpha_+-\alpha_-)\,d\tau
=\int_{1-\bar c}^{1+\bar c}\tau\,\alpha_-(\tau)\,d\tau.
\end{equation}
We can also write
\[
\int_0^\pi(\bar c-\cos\theta)^3\widetilde f\,d\theta
=\int_{1-\bar c}^{1+\bar c}\tau^3\,\alpha_-(\tau)\,d\tau
-\int_0^{1-\bar c}\tau^3(\alpha_+-\alpha_-)\,d\tau .
\]
On $[0,1-\bar c]$, we have $\alpha_+-\alpha_-\geq0$ and
$\tau^2\leq(1-\bar c)^2$, while
$\tau^2\geq(1-\bar c)^2$ on $[1-\bar c,1+\bar c]$. Hence
\[
\int_0^{1-\bar c}\tau^3(\alpha_+-\alpha_-)\,d\tau
\leq (1-\bar c)^2\int_0^{1-\bar c}\tau(\alpha_+-\alpha_-)\,d\tau
\stackrel{\eqref{eq:zm-tau}}{=}(1-\bar c)^2\int_{1-\bar c}^{1+\bar c}\tau\,\alpha_-\,d\tau
\leq \int_{1-\bar c}^{1+\bar c}\tau^3\,\alpha_-\,d\tau .
\]
Hence $\int_0^\pi(\bar c-\cos\theta)^3\widetilde f\,d\theta\mathrel{\geq} 0$, as required.
\end{proof}

In view of~\eqref{eq:cubic-moment-transform} and
Lemma~\ref{lem:monotone-cosine-skewness}, it remains to prove that
$\widetilde f$ is nonincreasing.
We differentiate~\eqref{eq:theta-change} and use
$\sin\theta=\sqrt{\Delta_0\Delta_\pi}/b$ together
with~\eqref{eq:cosine-density} to obtain
\begin{equation}\label{eq:eta-H-Z}
\widetilde f(\theta(x))=\frac{2}{\mathcal Z}\frac{H(x)}{Z(x)},
\qquad H(x):=\frac{e^{-v(x)}\sqrt{\Delta_0(x)\Delta_\pi(x)}}{\sin x}.
\end{equation}
Given that $\theta$ is increasing and $Z$ is positive and nondecreasing
by Lemma~\ref{lem:Z-bounds}, it is enough to prove that $H$ is
nonincreasing. We prove this property of $H$ by using the concavity of $u$ (specifically the chord bound~\eqref{eq:chord}) together with the monotonicity of $fZ$.

\begin{lemma}[Monotonicity of the new density]\label{lem:geometric-mean}
The function $\widetilde f$ given by~\eqref{eq:cosine-density}
is nonincreasing on $[0,\pi]$.
\end{lemma}
\begin{proof}
Since $\theta$ is increasing, by~\eqref{eq:eta-H-Z}
and Lemma~\ref{lem:Z-bounds} it suffices to prove that the function $H$ is nonincreasing on $(0,\pi)$.
By Lemma~\ref{lem:Z-bounds}, the functions $z=w_x$ and $Z=z/\sin x$
are positive on $(0,\pi)$, and $e^{-v}Z$ is nonincreasing. Factoring $\sin^2 x=(1-\cos x)(1+\cos x)$, we write
\[
H^2=\Bigl(e^{-v}\,\frac{\Delta_0}{1-\cos x}\Bigr)\Bigl(e^{-v}\,\frac{\Delta_\pi}{1+\cos x}\Bigr),
\]
so that $2\log H$ is the sum of the logarithms of these two factors.
Using $(\Delta_0)_x=z$ and $(\Delta_\pi)_x=-z$, we get
\[
\partial_x\log\Bigl(e^{-v}\frac{\Delta_0}{1-\cos x}\Bigr)=-u+\frac z{\Delta_0}-\frac{\sin x}{1-\cos x},\qquad
\partial_x\log\Bigl(e^{-v}\frac{\Delta_\pi}{1+\cos x}\Bigr)=-u-\frac z{\Delta_\pi}+\frac{\sin x}{1+\cos x}.
\]
We set
\begin{equation}\label{eq:rt-def}
r:=\cot\tfrac x2=\frac{\sin x}{1-\cos x}=\frac{1+\cos x}{\sin x},\qquad
t:=\tan\tfrac x2=\frac{\sin x}{1+\cos x}=\frac{1-\cos x}{\sin x}.
\end{equation}
Then $rt=1$, and we can write the logarithmic derivatives as
\[
\partial_x\log\Bigl(e^{-v}\frac{\Delta_0}{1-\cos x}\Bigr)=-u+\frac{z}{\Delta_0}-r,\qquad
\partial_x\log\Bigl(e^{-v}\frac{\Delta_\pi}{1+\cos x}\Bigr)=-u-\frac{z}{\Delta_\pi}+t.
\]
The first logarithmic derivative is nonpositive if and only if
$z/\Delta_0\leq u+r$, or equivalently, $\Delta_0\geq z/(u+r)$.
The second is nonpositive when $t\leq u$, since both $-u+t$ and
$-z/\Delta_\pi$ are nonpositive. When $t>u$, we need
$z/\Delta_\pi\geq t-u$, which is equivalent to $\Delta_\pi\leq z/(t-u)$.
So it suffices to prove
\begin{equation}\label{eq:a-pm-bounds}
\Delta_0(x)\geq\frac{z(x)}{u(x)+r(x)},\qquad
\Delta_\pi(x)\leq\frac{z(x)}{t(x)-u(x)}\ \text{on $\{t>u\}$.}
\end{equation}

To prove~\eqref{eq:a-pm-bounds}, we use the monotonicity of $e^{-v}Z$.
Since $(e^{-v}Z)_x\leq 0$, we have
\[
Z(y)\geq e^{v(y)-v(x)}Z(x)\text{ for }y\leq x,\qquad
Z(y)\leq e^{v(y)-v(x)}Z(x)\text{ for }y\geq x.
\]
Since $w_x=z=Z\sin x$, we multiply by $\sin y$ and integrate
over $[0,x]$ and $[x,\pi]$, respectively, to obtain
\begin{equation}\label{eq:a-ZI}
\Delta_0(x)\geq Z(x)\,I_0(x),\qquad \Delta_\pi(x)\leq Z(x)\,I_\pi(x),
\end{equation}
where
\[
I_0(x):=e^{-v(x)}\int_0^x e^{v(y)}\sin y\,dy,\qquad
I_\pi(x):=e^{-v(x)}\int_x^\pi e^{v(y)}\sin y\,dy.
\]
Differentiating these expressions and using $v_x=u$, we obtain the
first-order ODEs
\begin{equation}\label{eq:I-odes}
I_0'+uI_0=\sin x,\quad I_0(0)=0;\qquad
I_\pi'+uI_\pi=-\sin x,\quad I_\pi(\pi)=0.
\end{equation}
Since $z=Z\sin x$, we can deduce~\eqref{eq:a-pm-bounds}
from~\eqref{eq:a-ZI} by proving
\begin{equation}\label{eq:I-bounds}
I_0(x)\geq\frac{\sin x}{u(x)+r(x)},\qquad
I_\pi(x)\leq\frac{\sin x}{t(x)-u(x)}\ \text{on $\{t>u\}$.}
\end{equation}
Differentiating the definitions of $r$ and $t$ in~\eqref{eq:rt-def}, we get
\begin{equation}\label{eq:half-angle-diff}
\sin x\cdot r'=-r,\qquad \sin x\cdot t'=t.
\end{equation}

Let
\[
D_0:=\frac{\sin x}{u+r}.
\]
By~\eqref{eq:half-angle-diff} and~\eqref{eq:rt-def}, we have
$r'=-r/\sin x$ and $r\sin x=1+\cos x$.
Differentiating $D_0(u+r)=\sin x$ then gives
\begin{equation}\label{eq:D0-ode}
D_0'+uD_0=\sin x-\frac{u+\sin x\,u_x}{(u+r)^2}.
\end{equation}
In view of~\eqref{eq:chord},
the numerator $u+\sin x\,u_x$ in the right hand side of~\eqref{eq:D0-ode} is nonnegative.
Subtracting~\eqref{eq:D0-ode} from the equation for $I_0$ in~\eqref{eq:I-odes}
and multiplying by $e^v$, we then infer
\[
\bigl(e^v(I_0-D_0)\bigr)_x=e^v\bigl[(I_0-D_0)'+u(I_0-D_0)\bigr]\geq 0.
\]
Note that both $I_0$ and $D_0$ tend to zero as $x\to0^+$.
Indeed, since $u\to0$ and $r\sim2/x$, we see that $D_0\sim x^2/2$, whereas the definition of $I_0$ gives
$I_0\sim\int_0^x y\,dy=x^2/2$.
Hence $e^v(I_0-D_0)$ tends to zero at $0$ and is nondecreasing,
so $I_0\geq D_0$.

It remains to prove the second inequality in~\eqref{eq:I-bounds}.
Since $t'=t/\sin x$ by~\eqref{eq:half-angle-diff},
the chord bound~\eqref{eq:chord} implies
\[
\Bigl(\frac u t\Bigr)_x=\frac{u_x-U}{t}\leq 0,
\]
so $u/t$ is nonincreasing on $(0,\pi)$. Moreover, $u/t\to0$
as $x\to\pi^-$, since $u(\pi)=0$ and $t(x)\to+\infty$.
We deduce that $\{u/t<1\}=\{t>u\}$ is an interval $(x_*,\pi)$
for some $x_*\in[0,\pi)$. On this interval, set
\[
D_\pi:=\frac{\sin x}{t-u}.
\]
Differentiating $D_\pi(t-u)=\sin x$ and using $\sin x\cdot t'=t$
from~\eqref{eq:half-angle-diff} and $t\sin x=1-\cos x$
from~\eqref{eq:rt-def}, we obtain
\[
D_\pi'+uD_\pi=-\sin x+\frac{\sin x\,u_x-u}{(t-u)^2}.
\]
In view of ~\eqref{eq:chord}, we get $\sin x\,u_x\leq u$, so the
last term on the right hand side is nonpositive. Hence, subtracting the equation for $I_\pi$ in~\eqref{eq:I-odes}
from the equation for $D_\pi$ and multiplying by $e^v$, we obtain
\[
\bigl(e^v(D_\pi-I_\pi)\bigr)_x=e^v\bigl[(D_\pi-I_\pi)'+u(D_\pi-I_\pi)\bigr]\leq 0.
\]
As $x\to\pi^-$, we have $I_\pi\to0$ by  definition,
and $D_\pi\to0$ because $\sin x\to0$ and $t-u\to\infty$.
Integrating from $x$ to $\pi$, we conclude that $D_\pi\geq I_\pi$.

We have proved~\eqref{eq:I-bounds}, which gives~\eqref{eq:a-pm-bounds}.
It follows that $2(\log H)_x\leq 0$, so $H$ is nonincreasing.
Since $\theta$ is increasing and $Z$ is positive and nondecreasing
by Lemma~\ref{lem:Z-bounds}, we deduce from~\eqref{eq:eta-H-Z}
that $\widetilde f$ is nonincreasing.
\end{proof}

\begin{proposition}[Sign of the cubic moment]\label{prop:cubic}
For every $\zeta>0$,
\[
\int_\T w^3 f\geq 0.
\]
\end{proposition}
\begin{proof}
By Lemma~\ref{lem:geometric-mean}, $\widetilde f$ is nonincreasing.
Since $\widetilde f$ is a probability density and $b>0$,
Lemma~\ref{lem:monotone-cosine-skewness}
and~\eqref{eq:cubic-moment-transform} give the desired inequality.
\end{proof}

\subsection{\texorpdfstring{Conclusion}{Conclusion}}\label{subsec:conclusion}

\begin{proof}[Proof of Theorem~\ref{thm:main}]
By Lemma~\ref{lem:second-deriv},
\[
A''(\zeta)
=-\lambda\int_\T w^3 f
-3\int_\T w\,z^2 f.
\]
By Propositions~\ref{prop:cubic} and~\ref{prop:gradient}, respectively,
the first integral is nonnegative and the second is strictly positive.
Since $\lambda>0$, we deduce that $A''(\zeta)<0$ for $\zeta>0$,
and so $A$ is strictly concave. Part (i) then follows from the relation
$F_\kappa(\gamma)=\kappa A(4\gamma/\sigma^4)$ (see ~\eqref{eq:F-rescaling}).

By \cite[Prop.~6.4]{CCS}, every positive fixed point $\gamma_*$
of $F_\kappa$ yields a synchronized stationary equilibrium
$\mu^{\zeta_*}$, where $\zeta_*:=4\gamma_*/\sigma^4$, and every
synchronized stationary equilibrium is obtained in this way up to
rotation. To prove part~(ii), we thus consider the zeros of
$g_\kappa(\gamma):=F_\kappa(\gamma)-\gamma$.
From~\eqref{eq:F-properties}, $g_\kappa(0)=0$,
$g_\kappa'(0)=\kappa/\kappa_c-1$, and
$g_\kappa(\gamma)\leq\kappa-\gamma\to-\infty$ as $\gamma\to\infty$.
Since $g_\kappa$ is strictly concave on $[0,\infty)$ by part~(i)
and $g_\kappa(0)=0$, we also have
\begin{equation}\label{eq:tangent-bound}
g_\kappa(\gamma)<g_\kappa'(0)\,\gamma\qquad\text{for }\gamma>0.
\end{equation}
When $\kappa\leq\kappa_c$, the derivative $g_\kappa'(0)$
is nonpositive, and~\eqref{eq:tangent-bound} implies that
$g_\kappa<0$ on $(0,\infty)$. Thus $g_\kappa$ has no positive zero.

Now let $\kappa>\kappa_c$. We have $g_\kappa'(0)>0$,
and hence $g_\kappa(\gamma)>0$ for $\gamma>0$ small enough.
Since $g_\kappa(\gamma)\to-\infty$ as $\gamma\to\infty$, the intermediate
value theorem implies that $g_\kappa$ has a positive zero.
Two positive zeros would contradict strict concavity, because
$g_\kappa(0)=0$. Denote the unique positive zero by $\gamma_*(\kappa)$,
and let $\zeta_*(\kappa):=4\gamma_*(\kappa)/\sigma^4$.
Then $g_\kappa>0$ on $(0,\gamma_*(\kappa))$ and $g_\kappa<0$
on $(\gamma_*(\kappa),\infty)$, and by \cite[Prop.~6.4]{CCS},
$\mu_\kappa=\mu^{\zeta_*(\kappa)}$.

Strict concavity also gives $g_\kappa'(\gamma_*(\kappa))<0$,
since $g_\kappa$ vanishes at both $0$ and $\gamma_*(\kappa)$.
As $g_\kappa(\gamma)$ depends smoothly on both $\kappa$ and $\gamma$,
we may apply the implicit function theorem to $g_\kappa(\gamma)=0$
to conclude that $\gamma_*\in C^\infty((\kappa_c,\infty))$.
The same holds for $\zeta_*$, since $\zeta_*(\kappa)=4\gamma_*(\kappa)/\sigma^4$.
Moreover, by Proposition~\ref{prop:smooth-zeta} and~\eqref{eq:mu},
$f^\zeta$ depends smoothly on $\zeta$ in $C^\infty(\T)$.
Applying this to $\zeta=\zeta_*(\kappa)$ proves the claimed smooth
dependence of the equilibrium density on $\kappa$.

We finally prove that the equilibrium density converges to the
uniform density as $\kappa\downarrow\kappa_c$. Proposition~\ref{prop:smooth-zeta}
and~\eqref{eq:mu} imply that $f^\zeta\to f^0\equiv 1/(2\pi)$ in
$C^\infty(\T)$ as $\zeta\to 0$. It is therefore enough to prove
$\gamma_*(\kappa)\to 0$, since $\zeta_*(\kappa)=4\gamma_*(\kappa)/\sigma^4$.
Given $\epsilon>0$, we have $g_{\kappa_c}(\epsilon)<0$
by~\eqref{eq:tangent-bound}. By continuity, there exists $\delta>0$
such that $g_\kappa(\epsilon)<0$ for $\kappa\in(\kappa_c,\kappa_c+\delta)$.
On the other hand, $g_\kappa$ is nonnegative on $[0,\gamma_*(\kappa)]$,
so $\gamma_*(\kappa)<\epsilon$ whenever $\kappa\in(\kappa_c,\kappa_c+\delta)$.
Since $\epsilon>0$ was arbitrary, we finally conclude that
$\gamma_*(\kappa)\to 0$ as $\kappa\downarrow\kappa_c$, which completes the proof.
\end{proof}

\begin{remark}[Relation with the classical Kuramoto model]\label{rem:classical-limit}
The classical noisy Kuramoto model can be formally recovered from
the MFG by letting $\beta\to\infty$, with $K:=\kappa/\beta$ fixed. At a fixed
point of~\eqref{eq:self-consistency}, let $r:=\mu(\cos)$, so that
$\gamma=\beta Kr$. Dividing~\eqref{eq:HJB-intro} by $\beta$ then gives
\begin{equation}\label{eq:myopic-HJB}
\phi^\gamma-\frac{\sigma^2}{2\beta}\phi^\gamma_{xx}
+\frac{1}{2\beta}(\phi^\gamma_x)^2=-Kr\cos x.
\end{equation}
At the leading order, \eqref{eq:myopic-HJB} yields $\phi=-Kr\cos x$,
with the optimal feedback $-Kr\sin x$. Since $\mu(\sin)=0$, this is exactly the drift
\[
K\int_\T\sin(y-x)\,\mu(dy)=-Kr\sin x
\]
for the classical McKean--Vlasov Kuramoto dynamics. The limiting
equilibrium density is the von Mises law
$\exp(2Kr\cos x/\sigma^2)$. Let $I_0,I_1$ be the modified Bessel
functions. Then~\eqref{eq:self-consistency} becomes
the classical equation~\cite{ABPVS,BGP10}
\begin{equation}\label{eq:classical-self-consistency}
r=\frac{I_1(2Kr/\sigma^2)}{I_0(2Kr/\sigma^2)}.
\end{equation}

Concavity of $I_1/I_0$ implies uniqueness
of the positive solution of~\eqref{eq:classical-self-consistency},
analogously to how concavity of $F_\kappa$ does for the MFG.
This concavity, a classical result, can be recovered from Lemma~\ref{lem:monotone-cosine-skewness}.
Indeed, for $s\geq0$, the density
$\widetilde f(\theta)=e^{s\cos\theta}/(\pi I_0(s))$
is nonincreasing on $[0,\pi]$.
The mean of $\cos\theta$ under this density is $I_1(s)/I_0(s)$.
Its second derivative with respect to $s$ is the third mean-zero moment
of $\cos\theta$. By Lemma~\ref{lem:monotone-cosine-skewness},
this moment is nonpositive.
We thus recover the single-site
Griffiths--Hurst--Sherman inequality~\cite{GHS,EllisNewman}.

The rescaled
threshold $\kappa_c/\beta=\sigma^2+\sigma^4/(2\beta)$ converges to
the classical critical coupling $K_c=\sigma^2$.
\end{remark}


\begin{thebibliography}{ABPVS05}

\bibitem[ABPVS05]{ABPVS}
J.~A.~Acebr\'on, L.~L.~Bonilla, C.~J.~P\'erez Vicente, F.~Ritort, and
  R.~Spigler.
\newblock The {K}uramoto model: a simple paradigm for synchronization
  phenomena.
\newblock \emph{Reviews of Modern Physics} 77 (2005), no.\ 1, 137--185.

\bibitem[BGP10]{BGP10}
L.~Bertini, G.~Giacomin, and K.~Pakdaman.
\newblock Dynamical aspects of mean field plane rotators and the {K}uramoto model.
\newblock \emph{Journal of Statistical Physics} 138 (2010), no.\ 1--3, 270--290.

\bibitem[BGP14]{BGP}
L.~Bertini, G.~Giacomin, and C.~Poquet.
\newblock Synchronization and random long time dynamics for mean-field
  plane rotators.
\newblock \emph{Probability Theory and Related Fields} 160 (2014),
  no.\ 3--4, 593--653.

\bibitem[CD18]{CarmonaDelarue}
R.~Carmona and F.~Delarue.
\newblock \emph{Probabilistic Theory of Mean Field Games with
  Applications}.
\newblock Vol.~I: Mean Field FBSDEs, Control, and Games; Vol.~II: Mean
  Field Games with Common Noise and Master Equations.
\newblock Probability Theory and Stochastic Modelling, vols.\ 83--84.
  Springer, 2018.

\bibitem[CCS23]{CCS}
R.~Carmona, Q.~Cormier, and H.~M.~Soner.
\newblock Synchronization in a {K}uramoto mean field game.
\newblock \emph{Communications in Partial Differential Equations}
  48 (2023), no.\ 9, 1214--1244. arXiv:2210.12912.

\bibitem[CCS25]{CCS2}
R.~Carmona, Q.~Cormier, and H.~M.~Soner.
\newblock {K}uramoto mean field game with intrinsic frequencies.
\newblock Preprint, 2025. arXiv:2509.18000.

\bibitem[CC24]{CesaroniCirant}
A.~Cesaroni and M.~Cirant.
\newblock Stationary equilibria and their stability in a {K}uramoto {MFG}
  with strong interaction.
\newblock \emph{Communications in Partial Differential Equations}
  49 (2024), no.\ 1--2, 121--147. arXiv:2307.09305.

\bibitem[CG20]{CG}
R.~Carmona and C.~V.~Graves.
\newblock Jet lag recovery: synchronization of circadian oscillators
  as a mean field game.
\newblock \emph{Dynamic Games and Applications} 10 (2020), no.\ 1,
  79--99.

\bibitem[Cir19]{Cirant}
M.~Cirant.
\newblock On the existence of oscillating solutions in non-monotone
  mean-field games.
\newblock \emph{Journal of Differential Equations} 266 (2019), no.\ 12,
  8067--8093.

\bibitem[EN78]{EllisNewman}
R.~S.~Ellis and C.~M.~Newman.
\newblock Necessary and sufficient conditions for the {GHS} inequality with applications to analysis and probability.
\newblock \emph{Transactions of the American Mathematical Society} 237 (1978), 83--99.

\bibitem[GHS70]{GHS}
R.~B.~Griffiths, C.~A.~Hurst, and S.~Sherman.
\newblock Concavity of magnetization of an {I}sing ferromagnet in a positive external field.
\newblock \emph{Journal of Mathematical Physics} 11 (1970), no.\ 3, 790--795.

\bibitem[GPP12]{GPP}
G.~Giacomin, K.~Pakdaman, and X.~Pellegrin.
\newblock Global attractor and asymptotic dynamics in the {K}uramoto
  model for coupled noisy phase oscillators.
\newblock \emph{Nonlinearity} 25 (2012), no.\ 5, 1247--1273.

\bibitem[GT01]{GT}
D.~Gilbarg and N.~S.~Trudinger.
\newblock \emph{Elliptic Partial Differential Equations of Second Order}.
\newblock Reprint of the 1998 edition. Classics in Mathematics.
Springer, Berlin, 2001.

\bibitem[Har02]{Hartman}
P.~Hartman.
\newblock \emph{Ordinary Differential Equations}.
\newblock Second edition. Classics in Applied Mathematics, vol.~38. SIAM, Philadelphia, 2002.

\bibitem[HMC06]{HMC}
M.~Huang, R.~P.~Malham\'e, and P.~E.~Caines.
\newblock Large population stochastic dynamic games: closed-loop
  {M}c{K}ean--{V}lasov systems and the {N}ash certainty equivalence
  principle.
\newblock \emph{Communications in Information \& Systems} 6 (2006),
  no.\ 3, 221--252.

\bibitem[Hof26]{Hofer26}
F.~H\"ofer.
\newblock Convergence of potential mean-field games via {L}yapunov methods.
\newblock Preprint, 2026. arXiv:2604.17639.

\bibitem[HS25]{HS}
F.~H\"ofer and H.~M.~Soner.
\newblock Synchronization games.
\newblock \emph{Mathematics of Operations Research} 51 (2025),
  no.\ 2, 1443--1462. arXiv:2402.08842.

\bibitem[Kur75]{Kuramoto}
Y.~Kuramoto.
\newblock Self-entrainment of a population of coupled non-linear
  oscillators.
\newblock In \emph{International Symposium on Mathematical Problems in
  Theoretical Physics}, Lecture Notes in Physics, vol.\ 39, pages
  420--422. Springer, 1975.

\bibitem[LL07]{LL}
J.-M.~Lasry and P.-L.~Lions.
\newblock Mean field games.
\newblock \emph{Japanese Journal of Mathematics} 2 (2007), no.\ 1,
  229--260.

\bibitem[Str00]{Strogatz}
S.~H.~Strogatz.
\newblock From {K}uramoto to {C}rawford: exploring the onset of
  synchronization in populations of coupled oscillators.
\newblock \emph{Physica D: Nonlinear Phenomena} 143 (2000), no.\ 1--4,
  1--20.

\bibitem[YMMS12]{YMMS}
H.~Yin, P.~G.~Mehta, S.~P.~Meyn, and U.~V.~Shanbhag.
\newblock Synchronization of coupled oscillators is a game.
\newblock \emph{IEEE Transactions on Automatic Control} 57 (2012),
  no.\ 4, 920--935.

\end{thebibliography}
\end{document}